\documentclass[a4paper,12pt]{article}

\usepackage{ifthen}
\newboolean{annotate}

\setboolean{annotate}{false}

\usepackage[color]{changebar}
\cbcolor{red}

\newcommand{\annotate}[1]{\ifthenelse{\boolean{annotate}}{\cbstart \textcolor{red}{#1} \cbend}{#1}}
\newcommand{\annotateBar}[1]{\ifthenelse{\boolean{annotate}}{\cbstart #1 \cbend}{#1}}
\newcommand{\annotateText}[1]{\ifthenelse{\boolean{annotate}}{\textcolor{red}{#1}}{#1}}

\usepackage{graphicx}
\usepackage{amsmath,amssymb,amsfonts,amsthm}
\usepackage{hyperref}
\usepackage{booktabs}
\usepackage{xcolor}
\hypersetup{
	colorlinks,
	linkcolor={blue},
	citecolor={blue},
	urlcolor={blue}
}

\usepackage{authblk}
\theoremstyle{plain}
\newtheorem{theorem}{Theorem}
\newtheorem{lemma}{Lemma}
\newtheorem{corollary}{Corollary}

\theoremstyle{definition}

\theoremstyle{remark}

\newcommand{\reals}{\mathbb{R}}

\newcommand{\vect}[1]{\mbox{\boldmath$#1$}}

\DeclareMathOperator{\cav}{cav}
\DeclareMathOperator{\vex}{vex}
\DeclareMathOperator{\chgap}{chgap}
\DeclareMathOperator{\mcgap}{mcgap}
\DeclareMathOperator{\mcl}{mcl}
\DeclareMathOperator{\mcu}{mcu}
\DeclareMathOperator{\conv}{conv}

\newcommand{\prob}[1]{\textbf{P}\left(#1\right)}
\newcommand{\expect}[1]{\textbf{E}\left(#1\right)}

\title{Bounding the gap between the McCormick relaxation and the convex hull for bilinear functions}

\author[1]{Natashia Boland}
\author[1]{Santanu S. Dey}
\author[2]{Thomas Kalinowski}
\author[3]{Marco Molinaro}
\author[2]{Fabian Rigterink}

\affil[1]{Georgia Institute of Technology, Atlanta, USA} 
\affil[2]{University of Newcastle, Australia}
\affil[3]{PUC-Rio, Brazil}

\begin{document}
\maketitle
\begin{abstract}
  We investigate how well the graph of a bilinear function $b:[0,1]^n\to\reals$ can be approximated
  by its McCormick relaxation. In particular, we are interested in the smallest number $c$ such that
  the difference between the concave upper bounding and convex lower bounding functions obtained
  from the McCormick relaxation approach is at most $c$ times the difference between the concave and
  convex envelopes. Answering a question of Luedtke, Namazifar and Linderoth, we show that this
  factor $c$ cannot be bounded by a constant independent of $n$. More precisely, we show that for a
  random bilinear function $b$ we have asymptotically almost surely $c\geqslant\sqrt
  n/4$. \annotateText{On the other hand, we prove that $c\leqslant 600\sqrt{n}$, which improves the
    linear upper bound proved by Luedtke, Namazifar and Linderoth.} In addition, we
  \annotateText{present an alternative proof for a result of Misener, Smadbeck and Floudas
    characterizing functions $b$ for which the McCormick relaxation is equal to the convex hull.}
\end{abstract}

An important technique in global optimization is the construction of convex envelopes for nonconvex
functions over convex sets (see for instance~\cite{HorstTuy-2013-GlobaloptimizationDeterministic}),
and consequently, there has been a lot of work on such envelopes of special classes of
functions~\cite{Al-Khayyal.Falk_1983_JointlyConstrainedBiconvex,Crama_1993_Concaveextensionsnonlinear,Rikun_1997_ConvexEnvelopeFormula,Sherali_1997_Convexenvelopesmultilinear,TawarmalaniRichardXiong-2012-Explicitconvexand}.
Many modern global optimization
solvers~\cite{Belotti.etal_2009_Branchingandbounds,Sahinidis_1996_BARON:generalpurpose,Smith.Pantelides_1999_symbolicreformulationspatialbranch}
follow a general approach, proposed by McCormick~\cite{McCormick_1976_Computabilityglobalsolutions},
that is based on a linear relaxation for bilinear terms. Luedtke, Namazifar, and
Linderoth~\cite{Luedtke.etal_2012_Someresultsstrength} proved a number of \annotateText{statements} about the strength
of the resulting relaxations for multilinear functions. In this note we extend their results on
bilinear functions. In particular, we characterize the bilinear functions for which the McCormick
relaxation describes the convex hull, \annotateText{we improve the upper bound on this approximation ratio, and we
prove that our new bound is asymptotically tight, thus providing a negative answer to a question
from~\cite{Luedtke.etal_2012_Someresultsstrength}.}

Consider a bilinear function $b:[0,1]^n\to\reals$ given by
\[b(\vect x)=\sum_{ij\in E}a_{ij}x_ix_j\] 
with coefficients $a_{ij}\in\reals$, where $G=(V,E)$ is an undirected
graph with vertex set $V=\{1,\ldots,n\}$, and we write $ij$ for $\{i,j\}$. The graph of $b$ is the
set
\[B=\{(\vect x,z)\in [0,1]^n\times\reals\ :\ z=b(\vect x)\},\] 
and we are interested in relaxations
of the convex hull of $B$, which can be characterized as (see
~\cite{Rikun_1997_ConvexEnvelopeFormula})
\[\conv(B)=\left\{(\vect x,z)\in [0,1]^n\times\reals\ :\ \exists\vect\lambda\in\Delta_{2^n}\ \text{with }\vect
  x=\sum_{k=1}^{2^n}\lambda_k\vect x^k,\ z=\sum_{k=1}^{2^n}\lambda_kb(\vect x^k)\right\},\] where
$\vect x^1$, \ldots, $\vect x^{2^n}$ are the vertices of $[0,1]^n$ and
$\Delta_{2^n}=\{\vect\lambda\in[0,1]^{2^n}\ :\ \sum_{k=1}^{2^n}\lambda_k=1\}$ is the
$(2^n-1)$-simplex. The McCormick relaxation~\cite{McCormick_1976_Computabilityglobalsolutions}
approximates $B$ by introducing for each bilinear term $x_ix_j$ a new variable $y_{ij}$ together
with the constraints $0\leqslant y_{ij}\leqslant x_i$, $y_{ij}\leqslant x_j$ and $y_{ij}\geqslant
x_i+x_j-1$. More precisely, we define two convex polytopes $P=P(G)\subseteq\reals^{n+\lvert E\rvert}$ and
$Q=Q(b)\subseteq\reals^{n+1}$:
\begin{align*}
P &= \{(\vect x,\,\vect y)\in [0,1]^n\times[0,1]^{\annotateText{\lvert E\rvert}}\ :\ y_{ij}\leqslant x_i,\ y_{ij}\leqslant x_j,\ y_{ij}\geqslant
x_i+x_j-1\text{ for all }ij\in E\},\text{ and}\\
Q &= \left\{(\vect x,\,z)\in[0,1]^n\times\reals\ :\ \exists \vect y\in[0,1]^{\annotateText{\lvert E\rvert}} \text{ with } (\vect x,\vect y)\in P\text{ and }z=\sum_{ij\in E}a_{ij}y_{ij}\right\}.   
\end{align*}

\section{Main results}\label{sec:results}
We have $\conv(B)\subseteq Q$ and it is natural to ask how well $Q$ approximates $\conv(B)$. Following the
notation from~\cite{Luedtke.etal_2012_Someresultsstrength} we denote the concave and convex
envelopes of the graph of $b$ by $\cav[b]$ and $\vex[b]$, respectively, and the corresponding upper and
lower McCormick envelopes by $\mcu[b]$ and $\mcl[b]$, respectively. These envelopes are functions from $[0,1]^n$ to
$\reals$ defined by
\begin{align*}
  \cav[b](\vect x) & =\max\{z\ :\ (\vect x,z)\in\conv(B)\}, & \vex[b](\vect x) & =\min\{z\ :\ (\vect x,z)\in\conv(B)\},\\
  \mcu[b](\vect x) & =\max\{z\ :\ (\vect x,z)\in Q\}, & \mcl[b](\vect x) & =\min\{z\ :\ (\vect x,z)\in Q\}. 
\end{align*}
We call the corresponding differences \emph{convex hull gap}, denoted by $\chgap[b]$, and
\emph{McCormick gap}, denoted by $\mcgap[b]$, respectively. In other words,
\begin{align*}
  \chgap[b](\vect x) &= \cav[b](\vect x)-\vex[b](\vect x)\ \ \quad\text{and}&
  \mcgap[b](\vect x) &= \mcu[b](\vect x)-\mcl[b](\vect x).
\end{align*}
Our measure for the quality of $Q$ as an approximation of $\conv(B)$ is the number
\[c^*(b)=\inf\{c\in\reals\ :\ \mcgap[b](\vect x)\leqslant c\chgap[b](\vect x)\text{ for all }\vect
x\in \annotateText{[0,1]^n}\}.\]  
In~\cite{Luedtke.etal_2012_Someresultsstrength} it is proved that \annotateText{under the condition that all
nonzero coefficients are positive we have 
\[c^*(b)\leqslant 2-\frac{1}{\lceil\chi(G)/2\rceil},\]
where $\chi(G)$ is the chromatic number of the graph $G$.
For arbitrary coefficients, the much weaker bound $c^*(b)\leqslant n$ is established,} and it is left as an open question if
$c^*(b)$ can be bounded by a constant independent of $n$ \annotateText{in the general case}. We provide a negative answer to this
question by proving the following theorem.
\begin{theorem}\label{thm:asymptotics}
Let $G=(V,E)$ be the complete graph on the vertex set $V=\{1,\ldots,n\}$, and let $b(\vect
x)=\sum_{ij\in E}a_{ij}x_ix_j$ where the coefficients $a_{ij}$ are chosen independently and
uniformly at random from $\{1,-1\}$. For $\vect x=\left(1/2,1/2,\ldots,1/2\right)$ we have
\[\lim_{n\to\infty}\prob{\mcgap[b](\vect x)\geqslant\frac{\sqrt n}{4}\chgap[b](\vect x)}=1.\]
\end{theorem}
\annotateBar{Moreover, we show that $\sqrt{n}$ is the correct leading term for the asymptotics. 
\begin{theorem}\label{thm:upper_bound}
For every bilinear function $b:[0,1]^n\to\reals$, and every $\vect x\in[0,1]^n$,
\[\mcgap[b](\vect x)\leqslant 600\sqrt n\chgap[b](\vect x).\]
\end{theorem}
In order to prove Theorem~\ref{thm:upper_bound} we establish the following discrepancy result which
might be of independent interest.
\begin{theorem}\label{thm:large_cuts}
Let $G=(V,E)$ be the complete graph on the vertex set $V=\{1,\ldots,n\}$, and let $\vect
a=(a_{ij})\in\reals^{n(n-1)/2}$ be a vector of edge weights. There exists a set $U\subseteq V$ such that
\[\left\lvert \sum_{ij\in\delta(U)}a_{ij}\right\rvert\geqslant\frac{1}{600\sqrt
  n}\sum_{ij\in E}\lvert a_{ij}\rvert\]
where the sum on the LHS is over the set $\delta(U)\subset E$ of edges with exactly one
vertex in $U$.
\end{theorem}
Finally, we give a characterization of the functions $b$ with with $Q=\conv(B)$.} Let us call an
edge $ij\in E$ positive if $a_{ij}>0$ and negative if $a_{ij}<0$. Without loss of generality we
assume that $a_{ij}\neq 0$ for all $ij\in E$, so every edge is either positive or negative. \annotateText{The
following theorem is a direct consequence of Theorem 3.10
in~\cite{Misener.etal_2014_Dynamicallygeneratedcutting} which states that the McCormick inequalities
are sufficient to describe the convex envelope of the graph of $b$ if and only if the number of
positive edges in every cycle is even. In order to capture the concave envelope as well we just need to
ensure that every cycle also contains an even number of negative edges.}
\begin{theorem}\label{thm:equality}
We have $Q=\conv(B)$ if and only if every cycle in $G$ has an even number of positive edges and
an even number of negative edges.
\end{theorem}
As a consequence, we can have $Q=\conv(B)$ only if $G$ is bipartite. Moreover, if $G$ is a
forest then $Q=\conv(B)$ for every choice of the coefficients $a_{ij}$, but as soon as $G$ contains
a cycle we can write down coefficients $a_{ij}$ such that $Q\neq\conv(B)$.

Our proofs are based on the following ideas from~\cite{Luedtke.etal_2012_Someresultsstrength}.
For a vector $\vect x\in [0,1]^n$, let $T_f=T_f(\vect x)\subseteq V$ be the set of indices of fractional
values, i.e., $T_f=\{i\in V\ :\ 0<x_i<1\}$. The proof of $\mcgap[b](\vect x)\leqslant
n\chgap[b](\vect x)$ for all $\vect x\in H$ in~\cite{Luedtke.etal_2012_Someresultsstrength} proceeds
in 3 steps.
\begin{enumerate}
\item $\displaystyle\vect x\in\{0,1/2,1\}^n\implies\mcgap[b](\vect x)=\frac12\displaystyle\sum_{ij\in E,\, i,j \in T_f}\lvert a_{ij}\rvert$.
\item $\displaystyle\vect x\in\{0,1/2,1\}^n\implies\chgap[b](\vect x)\geqslant\frac{1}{2\lvert T_f\rvert}\displaystyle\sum_{ij\in E,\, i,j \in T_f}\lvert a_{ij}\rvert$.
\item The function $c\chgap[b](\vect x)-\mcgap[b](\vect x)$ is minimized at some $\vect x\in\{0,1/2,1\}^n$.
\end{enumerate}
We will show that the argument for step 2 can be modified to provide a lower bound for
$\chgap[b](\vect x)$ in terms of the difference between the maximum and the minimum cut in the
subgraph of $G$ induced by $T_f$. Theorem~\ref{thm:asymptotics} then follows by applying the
Chernoff inequality, \annotateText{Theorem~\ref{thm:large_cuts}, and consequently
  Theorem~\ref{thm:upper_bound}, is proved using probabilistic arguments that have been
  developed in the context of studying the discrepancy of graphs~\cite{Bollobas.Scott_2006_DiscrepancyGraphsHypergraphs,Erdoes.Spencer_1971_Imbalanceskcolorations,Erdos.etal_1988_Cuttinggraphtwo},} and Theorem~\ref{thm:equality} is a consequence of the
observation that the difference between the maximum and the minimum cut is equal to the sum of the
absolute values of all weights if and only if the sets of positive and negative edges form two cuts
of the graph.

\section{Proofs of the theorems}\label{sec:proofs}
\annotateText{\subsection{Characterizing the convex hull gap in terms of cuts}}\label{subsec:preliminaries}
Let $G=(V,E)$ be a graph with vertex set $V=[n]$. We use the following notation from~\cite{Luedtke.etal_2012_Someresultsstrength}.
\begin{itemize}
\item For $X\subseteq V$, $\gamma(X)$ is the set of edges with both vertices in $X$.
\annotateText{\item For $X\subseteq V$, $\delta(X)$ is the set of edges with exactly one vertex in $X$.}
\item For $X,Y\subseteq V$ with $X\cap Y=\emptyset$, $\delta(X,Y)$ is the set of edges with one
  vertex in $X$ and one vertex in $Y$.
\item For $i\in V$, $\mathcal S_i$ is the collection of vertex sets that contain $i$, i.e.,
  $\mathcal S_i=\{W\subseteq V\ :\ i\in W\}$.
\item For $Z\subseteq E$, we put $a(Z)=\sum_{ij\in Z}a_{ij}$.
\end{itemize}
We denote the maximum and the minimum weight of a cut in the subgraph induced by
$X\subseteq V$ with $\mu^+(X)$ and $\mu^-(X)$, i.e.,
\begin{align*}
  \mu^+(X)=\max\left\{\sum_{ij\in\delta(U_1,U_2)}a_{ij}\ :\ U_1\cup U_2=X,\ U_1\cap
    U_2=\emptyset\right\},\\
\mu^-(X)=\min\left\{\sum_{ij\in\delta(U_1,U_2)}a_{ij}\ :\ U_1\cup U_2=X,\ U_1\cap U_2=\emptyset\right\}.
\end{align*}
We identify $\{0,1\}^n$ with the power set of $V$ in the natural way: $\vect x\in\{0,1\}^n$ is
identified with the set $\{i\ :\ x_i=1\}$.
We start by establishing that the upper bound for $\chgap[b](\vect x)$ in terms of cuts in induced
subgraphs of $G$, proved in~\cite{Luedtke.etal_2012_Someresultsstrength} (Lemma 3.10), is tight.
\begin{lemma}\label{lem:chgap}
  Let $\vect x\in\{0,1/2,1\}^n$ and put $T_1=\{i\in V\ :\ x_i=1\}$ and $T_f=\{i\in V\ :\
  x_i=1/2\}$. Then
  \begin{align}
    \vex[b](\vect x) &=
    a(\gamma(T_1))+\frac12a(\delta(T_1,T_f))+\frac12a(\gamma(T_f))-\frac12\mu^+(T_f),\label{eq:vex_bound}\\
    \cav[b](\vect x) &=
    a(\gamma(T_1))+\frac12a(\delta(T_1,T_f))+\frac12a(\gamma(T_f))-\frac12\mu^-(T_f),\label{eq:cav_bound}\\
    \chgap[b](\vect x) &= \frac12\left(\mu^+(T_f)-\mu^-(T_f)\right).\label{eq:chgap_bound}
  \end{align}
\end{lemma}
\begin{proof}
  We start by writing $\vex[b](\vect x)$ as follows:
  \[\vex[b](\vect x)=\min\left\{\sum_{X\subseteq T_f}\lambda_Xa(\gamma(X\cup T_1))\ :\ \sum_{X\subseteq
      T_f}\lambda_X=1,\ \sum_{X\in\mathcal S_i}\lambda_X=1/2\ \forall i\in T_f,\ \vect\lambda\geqslant\vect 0\right\}.\]
Now 
\[a(\gamma(X\cup T_1))=a(\gamma(T_1))+a(\delta(T_1,X))+a(\gamma(X)),\]
and, for any $\vect\lambda$ satisfying $\sum_{X\in\mathcal S_i}\lambda_X=1/2$ for all $i\in T_f$, we have that
\[
  \sum_{X\subseteq T_f} \lambda_X a(\delta(T_1,X)) = \sum_{X\subseteq T_f}\lambda_X\sum_{i\in X}\sum_{j\in T_1,\, ij\in E} a_{ij} = \sum_{j\in T_1} \sum_{i\in T_f,\, ij\in E} \sum_{X\in\mathcal S_i} \lambda_X a_{ij} = \frac{1}{2}a(\delta(T_1,T_f)).
\]
Thus
\[\vex[b](\vect x)=a(\gamma(T_1))+\frac12a(\delta(T_1,T_f))+M,\]
where
\[M=\min\left\{\sum_{X\subseteq T_f}\lambda_Xa(\gamma(X))\ :\ \sum_{X\subseteq
      T_f}\lambda_X=1,\ \sum_{X\in\mathcal S_i}\lambda_X=1/2\ \forall i\in T_f,\ \vect\lambda \geqslant\vect 0\right\}.\]
As in the proof of Lemma 3.10 in~\cite{Luedtke.etal_2012_Someresultsstrength}, we can set
$\lambda_{U_1}=\lambda_{U_2}=1/2$ for a maximum cut $(U_1,U_2)$ in the subgraph induced by $T_f$, which yields 
\[M\leqslant
\frac12\left[a(\gamma(U_1))+a(\gamma(U_2))\right]=\frac12\left[a(\gamma(T_f))-\mu^+(T_f)\right].\]
In order to prove that this bound is tight, we look at the dual
\[M=\max\left\{y+\frac12\sum_{i\in T_f}z_i\ :\ y+\sum_{i\in X}z_i\leqslant a(\gamma(X))\ \forall X\subseteq T_f\right\}.\]
Setting $y=-\mu^+(T_f)/2$ and
\[z_i=\frac12\sum_{j\in T_f\,:\,ij\in E}a_{ij}\qquad\text{for }i\in T_f\]
we get a feasible solution, because for every $X\subseteq T_f$ we have
\[y+\sum_{i\in X}z_i=-\frac12\mu^+(T_f)+\frac12a(\delta(X,T_f\setminus X))+a(\gamma(X))\leqslant
a(\gamma(X)).\]
Since the objective value
\[y+\frac12\sum_{i\in T_f}z_i=-\frac12\mu^+(T_f)+\frac14\sum_{i\in T_f}\sum_{j\in T_f\,:\,ij\in
  E}a_{ij}=\frac12\left(a(T_f)-\mu^+(T_f)\right)\]
is equal to the upper bound for $M$ we have proved that $M$ is equal to this value, and this
concludes the proof of~\eqref{eq:vex_bound}. For~\eqref{eq:cav_bound} we use the same method to get
\[\cav[b](\vect x)=a(\gamma(T_1))+\frac12a(\delta(T_1,T_f))+M'\annotateText{,}\]
where $M'$ is characterized by 
\begin{align*}
  M' &=\max\left\{\sum_{X\subseteq T_f}\lambda_Xa(\gamma(X))\ :\ \sum_{X\subseteq
      T_f}\lambda_X=1,\ \sum_{X\in\mathcal S_i}\lambda_X=\frac12\ \forall i\in T_f,\ \vect{\lambda}\geqslant\vect 0\right\}\\
 &=\min\left\{y+\frac12\sum_{i\in T_f}z_i\ :\ y+\sum_{i\in X}z_i\geqslant a(\gamma(X))\ \forall X\subseteq T_f\right\}.
\end{align*}
Taking a minimum cut $(U_1',U_2')$ we get a primal solution $\lambda_{U_1'}=\lambda_{U_2'}=1/2$
and a corresponding dual solution $y=-\mu^-(T_f)/2$, 
\[z_i=\frac12\sum_{j\in T_f\,:\,ij\in E}a_{ij}\qquad\text{for }i\in T_f.\]
Finally, \eqref{eq:chgap_bound} follows by taking the difference of ~\eqref{eq:vex_bound}
and~\eqref{eq:cav_bound}. 
\end{proof}
By Lemma 3.9 from~\cite{Luedtke.etal_2012_Someresultsstrength}, we have $\mcgap[b](\vect
x)=\tfrac12\sum_{ij\in\gamma(T_f)}\lvert a_{ij}\rvert$ for all $\vect x\in\{0,1/2,1\}^n$, and
using the convexity argument from the proof of Theorem 3.12
in~\cite{Luedtke.etal_2012_Someresultsstrength} we get the following corollary.
\begin{corollary}\label{cor:zero_half}
Let $c$ be a number such that $\sum_{ij\in\gamma(X)}\lvert a_{ij}\rvert\leqslant
c\left(\mu^+(X)-\mu^-(X)\right)$ for all $X\subseteq V$. Then for all $\vect x\in[0,1]^n$, $\mcgap[b](\vect x)\leqslant c\chgap[b](\vect x)$.\qed
\end{corollary}
\annotateText{\subsection{The lower bound}}\label{subsec:lower_bound}
\begin{proof}[Proof of Theorem~\ref{thm:asymptotics}] 
  Let $G=(V,E)$ be the complete graph on the vertex set $V=\{1,\ldots,n\}$ and consider the bilinear
  function is
  \[b(\vect x)=\sum_{ij\in E}a_{ij}x_ix_j\] 
where the coefficients $a_{ij}$ are randomly chosen from
  $\{1,-1\}$ (independently and uniformly). 
  Using the Chernoff inequality \annotateText{and the fact that $\delta(U_1,U_2) \leqslant n^2/4$ for every cut$(U_1,U_2)$}, we have that,
\[\prob{\left\lvert\sum_{ij\in\delta(U_1,U_2)}a_{ij}\right\rvert>0.6n^{3/2}}<2e^{-0.72n}.\]
Taking the union bound over all $2^{n-1}$ cuts gives
\[\prob{-0.6 n^{3/2}\leqslant\sum_{ij\in\delta(U_1,U_2)}a_{ij}\leqslant 0.6n^{3/2}\text{ for
    all cuts }(U_1,U_2)}\geqslant 1-2^{n}e^{-0.72n}\annotateText{,}\]
which tends to $1$ as $n\to\infty$. So
\[\lim_{n\to\infty}\prob{\mu^+(V)-\mu^-(V)\leqslant 1.2n^{3/2}}=1\annotateText{,}\]
and consequently, for $\vect x=\left(1/2,1/2,\ldots,1/2\right)$, with probability
tending to 1 as $n\to\infty$,
\[\mcgap[b](\vect x)=\frac{\lvert E\rvert}{2}=\frac{n(n-1)}{4}>\frac{\sqrt n}{4}0.6n^{3/2}\geqslant\frac{\sqrt n}{4}\chgap[b](\vect x).\qedhere\]
\end{proof}
\annotateBar{Theorem~\ref{thm:asymptotics} ensures that there are many functions with a large ratio between the
McCormick gap and the convex hull gap. Next we construct an explicit example for every $n$. We
define a bilinear function $b:[0,1]^n\to\reals$ as follows. Let $k=\lceil\log_2(n)\rceil$. With vertex
$i\in V=\{1,\ldots,n\}$ we associate the vector $\vect i=(i_1,\ldots,i_k)\in\{0,1\}^k$ of the digits
of $i-1$ in binary representation, i.e., $i-1=i_12^0+i_22^1+\cdots+i_k2^{k-1}$, and we put
$a_{ij}=(-1)^{\langle\vect i,\vect j\rangle}$, where $\langle\cdot,\cdot\rangle$ is the standard
scalar product, $\langle\vect i,\vect j\rangle=i_1j_1+\cdots+i_kj_k$. The following lemma is a
standard discrepancy result (see for instance Chapter 10
in~\cite{Chazelle-2000-discrepancymethodRandomness}), but for convenience we include the short
proof.
\begin{lemma}\label{lem:discrepancy}
We have $\mu^+(V)\leqslant (n^{3/2})/\sqrt2$ and $\mu^-(V)\geqslant -(n^{3/2})/\sqrt2$. 
\end{lemma}
\begin{proof}
Let $H$ be the $2^k\times 2^k$ matrix with rows and columns indexed by binary strings of length $k$
with $H_{ij}=(-1)^{\langle i,j\rangle}$. Then $H$ is a Hadamard matrix, i.e., $H^TH=2^kI$ where $I$
is the identity matrix of size $2^k\times 2^k$. Therefore, $\lVert H\vect v\rVert_2\leqslant
2^{k/2}\lVert\vect v\rVert_2$ for every $\vect v$. The vertices in $V$ correspond to
the first $n$ rows and columns of $H$, and therefore we can identify a subset $U\subseteq V$ with a vector
$\vect u\in\{0,1\}^{2^k}$. For a cut $(U,V\setminus U)$, let $\vect w$ be the vector corresponding to $V\setminus U$. We can
bound the weight of this cut by
\[\left\lvert\sum_{ij\in\delta(U)}a_{ij}\right\rvert=\left\lvert\sum_{i\in U}\sum_{j\in V\setminus U}(-1)^{\langle\vect i,\vect
  j\rangle}\right\rvert=\left\lvert\vect u^TH\vect w\right\rvert\leqslant\lVert\vect u\rVert_2\lVert
H\vect w\rVert_2\leqslant 2^{k/2}\lVert\vect u\rVert_2\lVert
\vect w\rVert_2.\]
Now $(u_1+\cdots+u_{2^k})+(w_1+\cdots+w_{2^k})=n$, and the AM-GM inequality yields 
\begin{multline*}
\lVert\vect u\rVert_2\lVert\vect w\rVert_2=\sqrt{\left(u^2_1+\cdots+u^2_{2^k}\right)\left(w^2_1+\cdots+w^2_{2^k}\right)}
\leqslant \frac{\left(u^2_1+\cdots+u^2_{2^k}\right)+\left(w^2_1+\cdots+w^2_{2^k}\right)}{2}\\ =\frac{\left(u_1+\cdots+u_{2^k}\right)+\left(w_1+\cdots+w_{2^k}\right)}{2} = n/2
\end{multline*}
Consequently,
\[\left\lvert\sum_{ij\in\delta(U)}a_{ij}\right\rvert^2\leqslant 2^{k}n^2/4\leqslant n^3/2.\qedhere\]
\end{proof}

From Lemmas~\ref{lem:chgap} and~\ref{lem:discrepancy} it follows that $\displaystyle\chgap[b](1/2,\ldots,1/2)\leqslant
\frac{n^{3/2}}{\sqrt 2}$, and therefore
\[\mcgap[b](1/2,\ldots,1/2)=\frac{n(n-1)}{4}\geqslant \frac{\sqrt2}{4}\left(\sqrt{n}-\frac{1}{\sqrt{n}}\right)\chgap[b](1/2,\ldots,1/2).\]
So for $n\geqslant 18$ we have $\displaystyle\mcgap[b](1/2,\ldots,1/2)\geqslant\frac{\sqrt{n}}{3}\chgap[b](1/2,\ldots,1/2)$.

\annotateText{\subsection{The upper bound}}\label{subsec:upper_bound}
The unit weight case of Theorem~\ref{thm:large_cuts} has been proved
in~\cite{Erdos.etal_1988_Cuttinggraphtwo}, and here we extend this argument to the general case. We start with a
partition $V=L\cup R$ such that
\begin{align}\label{eq:half}
\sum_{ij\in\delta(L,R)}\lvert a_{ij}\rvert\geqslant\frac12\sum_{ij\in E}\lvert a_{ij}\rvert.
\end{align}
To see why such a partition exists, consider any random partition of vertices into two subsets, where with equal probability each vertex is assigned to any one of the subsets. Taking the edge weights to be $\lvert a_{ij}\rvert$, the expected value of the resulting cut is $\frac12\sum_{ij\in E}\lvert a_{ij}\rvert$. Therefore, there exists a specific partition $V=L\cup R$ which satisfies (\ref{eq:half}).

Now we choose a random subset $S\subseteq L$ ($\prob{i\in S}=1/2$ for every $i\in L$ and these
events are independent). 
\begin{lemma}\label{lem:anticoncentration}
For every $j\in R$, 
\[\prob{\left\lvert\sum_{i\in S}a_{ij}\right\rvert\geqslant \frac{1}{4}\left(\sum_{i\in
        L}a_{ij}^2\right)^{1/2}}\geqslant \frac{1}{24}.\]
\end{lemma}
\begin{proof}
Fix $j\in R$, and let $X_i$ for $i\in L$ be the random variable defined by $X_i=1$ if
$i\in S$ and $X_i=-1$ if $i\not\in S$, so that
\[\sum_{i\in S}a_{ij}=\frac12\sum_{i\in L}a_{ij}+\frac12\sum_{i\in L}a_{ij}X_i.\]
For $Z=\left(\sum_{i\in L}a_{ij}X_i\right)^2$, we have $\displaystyle\expect{Z}=\sum_{i\in L}a_{ij}^2$, and therefore
\begin{equation}\label{eq:P-Z}
\prob{Z\geqslant\frac12\sum_{i\in L}a_{ij}^2}\geqslant\frac14\frac{\left(\sum_{i\in
      L}a_{ij}^2\right)^2}{\expect{Z^2}}  
\end{equation}
by the Paley-Zygmund inequality. From the Khintchine inequality with the Haagerup
bounds~\cite{haagerup1981best,nazarov2000ball} it follows that
\[\expect{Z^2}=\expect{\left(\sum_{i\in L}a_{ij}X_i\right)^4}\leqslant 3\left(\sum_{i\in
    L}a_{ij}^2\right)^2,\]
hence~\eqref{eq:P-Z} implies 
\begin{multline*}\label{eq:anticoncentration}
\prob{\sum_{i\in L}a_{ij}X_i\geqslant\frac{1}{\sqrt2}\left(\sum_{i\in
      L}a^2_{ij}\right)^{1/2}}=\prob{\sum_{i\in L}a_{ij}X_i\leqslant-\frac{1}{\sqrt2}\left(\sum_{i\in
      L}a^2_{ij}\right)^{1/2}}\\
=\frac12
\prob{Z\geqslant\frac12\sum_{i\in L}a_{ij}^2}\geqslant\frac{1}{24}.
\end{multline*}
This gives the implications
\begin{align*}
  \sum_{i\in L}a_{ij}\geqslant 0 &\implies \prob{\sum_{i\in S}a_{ij}\geqslant\frac{1}{2\sqrt2}\left(\sum_{i\in
      L}a^2_{ij}\right)^{1/2}}\geqslant\frac{1}{24},\\
  \sum_{i\in L}a_{ij}\leqslant 0 &\implies \prob{\sum_{i\in S}a_{ij}\leqslant-\frac{1}{2\sqrt2}\left(\sum_{i\in
      L}a^2_{ij}\right)^{1/2}}\geqslant\frac{1}{24},
\end{align*}
and thus concludes the proof of the lemma (using $1/4<1/(2\sqrt2)$).
\end{proof}
From Lemma~\ref{lem:anticoncentration} and Cauchy-Schwarz we obtain
\begin{multline*}
\expect{\sum_{j\in R}\left\lvert\sum_{i\in S}a_{ij}\right\rvert}\geqslant\frac{1}{96}\sum_{j\in
  R}\left(\sum_{i\in L}a_{ij}^2\right)^{1/2}\geqslant \frac{1}{96}\sum_{j\in
  R}\left(\frac{1}{\lvert L\rvert^{1/2}}\sum_{i\in L}\lvert a_{ij}\rvert\right)\\
\geqslant \frac{1}{96\sqrt{n}}\sum_{i\in L}\sum_{j\in R}\lvert a_{ij}\rvert
\geqslant \frac{1}{200\sqrt{n}}\sum_{ij\in E}\lvert a_{ij}\rvert,
\end{multline*}
where the last inequality follows from (\ref{eq:half}).
This implies that there exists a set $S\subseteq L$ with
\begin{equation}\label{eq:good_S}
\sum_{j\in R}\left\lvert\sum_{i\in S}a_{ij}\right\rvert\geqslant\frac{1}{200\sqrt{n}}\sum_{ij\in
  E}\lvert a_{ij}\rvert.  
\end{equation}
Fix such a set $S$ and define the sets
\begin{align*}
  R_+&=\left\{j\in R\ :\ \sum_{i\in S}a_{ij}\geqslant 0\right\}, & R_-=\left\{j\in R\ :\ \sum_{i\in S}a_{ij}< 0\right\}.
\end{align*}
Then 
\[\sum_{j\in R}\left\lvert\sum_{i\in S}a_{ij}\right\rvert=\sum_{j\in R_+}\sum_{i\in
  S}a_{ij}-\sum_{j\in R_-}\sum_{i\in S}a_{ij},\]
and it follows from~\eqref{eq:good_S} that
\[\max\left\{\sum_{j\in R_+}\sum_{i\in
  S}a_{ij},\ -\sum_{j\in R_-}\sum_{i\in
  S}a_{ij}\right\}\geqslant\frac{1}{400\sqrt n}\sum_{ij\in
  E}\lvert a_{ij}\rvert.\]
Without loss of generality, we assume that the maximum is obtained by the first term, i.e.,
\[\sum_{j\in R_+}\sum_{i\in
  S}a_{ij}\geqslant\frac{1}{400\sqrt n}\sum_{ij\in
  E}\lvert a_{ij}\rvert.\]
We conclude the proof of Theorem~\ref{thm:large_cuts} as suggested
in~\cite{__Combinatorialdiscrepancysystem}. Let $W=V\setminus(S\cup R_+)$ and distinguish three cases.
\begin{description}
\item[Case 1.] If $\displaystyle\sum_{ij\in\delta(S,W)}a_{ij}\geqslant -\frac{1}{1200\sqrt n}\sum_{ij\in
    E}\lvert a_{ij}\rvert$ then we can take $U=S$:
\[\sum_{ij\in\delta(S)}a_{ij}=\sum_{ij\in\delta(S,R_+)}a_{ij}+\sum_{ij\in\delta(S,W)}a_{ij}\geqslant\left(\frac{1}{400}-\frac{1}{1200}\right)\frac{1}{\sqrt{n}}\sum_{ij\in
    E}\lvert a_{ij}\rvert.\]
\item[Case 2.] If $\displaystyle\sum_{ij\in\delta(R_+,W)}a_{ij}\geqslant -\frac{1}{1200\sqrt n}\sum_{ij\in
    E}\lvert a_{ij}\rvert$ then we can take $U=R_+$:
\[\sum_{ij\in\delta(R_+)}a_{ij}=\sum_{ij\in\delta(S,R_+)}a_{ij}+\sum_{ij\in\delta(R_+,W)}a_{ij}\geqslant\left(\frac{1}{400}-\frac{1}{1200}\right)\frac{1}{\sqrt{n}}\sum_{ij\in
    E}\lvert a_{ij}\rvert.\]
\item[Case 3.] If $\displaystyle\max\left\{\sum_{ij\in\delta(S,W)}a_{ij},\,\sum_{ij\in\delta(R_+,W)}a_{ij}\right\}< -\frac{1}{1200\sqrt n}\sum_{ij\in
    E}\lvert a_{ij}\rvert$ then we can take $U=S\cup R_+$:
\[\sum_{ij\in\delta(S\cup R_+)}a_{ij}=\sum_{ij\in\delta(S,W)}a_{ij}+\sum_{ij\in\delta(R_+,W)}a_{ij}<-\frac{1}{600\sqrt{n}}\sum_{ij\in
    E}\lvert a_{ij}\rvert.\]
\end{description}
\begin{proof}[Proof of Theorem~\ref{thm:upper_bound}]
  Applying Theorem~\ref{thm:large_cuts} to the subgraph induced by a vertex set $X\subseteq V$, yields
\[\mu^+(X)-\mu^-(X)\geqslant\frac{1}{600\sqrt{\lvert X\rvert}}\sum_{ij\in\gamma(X)}\lvert
a_{ij}\rvert\geqslant\frac{1}{600\sqrt{n}}\sum_{ij\in\gamma(X)}\lvert a_{ij}\rvert,\]
and now Corollary~\ref{cor:zero_half} implies the statement of the theorem.  
\end{proof}}

\annotateText{\subsection{Characterization of equality}}\label{subsec:equality} 
\annotateText{As
  mentioned in Section~\ref{sec:results}, Theorem~\ref{thm:equality} is a direct consequence of
  Theorem 3.10 in~\cite{Misener.etal_2014_Dynamicallygeneratedcutting}. We include the following short proof
  in order to show how this result can be derived from the correspondence between the convex hull gap and the
  range of cut weights in the graph~$G$.}
\begin{proof}[Proof of Theorem~\ref{thm:equality}]
  Suppose that every cycle in $G$ has an even number of positive edges and an even number of
  negative edges. Now let $X\subseteq V$ be any vertex set. We introduce two equivalence relations,
  $\sim_1$ and $\sim_2$, on $X$. For the first, we put $i\sim_1 j$ if $G$ contains a path between
  $i$ and $j$ consisting of positive edges. Similarly, we put $i\sim_2 j$ if $G$ contains a path
  between $i$ and $j$ consisting of negative edges. Let $G_1$ and $G_2$ be the quotient graphs,
  i.e., the vertices of $G_k$ ($k=1,2$) are the equivalence classes for $\sim_k$ and there is an
  edge between two classes $[i]$ and $[j]$ in $G_k$ if there is an edge in $G$ between any element
  of $[i]$ and any element of $[j]$. Note that the edges in $G_1$ correspond to negative edges of
  $G$, and the edges in $G_2$ correspond to positive edges of $G$. If every cycle in $G$ contains an
  even number of positive and negative edges, then $G_1$ and $G_2$ are bipartite. The partition of
  $G_1$ induces a partition $X=U_1\cup U_2$ such that $\delta(U_1,U_2)$ is the set of negative edges
  in $\gamma(X)$, and the partition of $G_2$ induces a partition $X=U'_1\cup U'_2$ such that
  $\delta(U'_1,U'_2)$ is the set of positive edges in $\gamma(X)$. Consequently,
  $\mu^+(X)-\mu^-(X)=\sum_{i\in\gamma(X)}\left\lvert a_{ij}\right\rvert$, and, since $X\subseteq V$
  was chosen arbitrarily, it follows, by Corollary~\ref{cor:zero_half}, that $\mcgap[b](\vect
  x)=\chgap[b](\vect x)$ for all $\vect x\in[0,1]^n$.

  Conversely, suppose that there exists a cycle that has an odd number of negative edges. Then any
  cut of $G$ that contains all negative edges in the graph, i.e., that contains the set $E^-=\{ij\in
  E\ :\ a_{ij}<0\}$, must contain at least one positive edge. This implies $\mu^-(V)>\sum_{ij\in
    E^-}a_{ij}$. 
So
\[\mu^+(V)-\mu^-(V)<\sum_{ij\in E}\left\lvert a_{ij}\right\rvert,\]
and consequently, by Lemma~\ref{lem:chgap}, $\chgap[b](1/2,\ldots,1/2)<\mcgap[b](1/2,\ldots,1/2)$. The
argument for a cycle with an odd number of positive edges is similar.
\end{proof} 
Theorem~\ref{thm:equality} implies that for functions without negative coefficients we have
$Q=\conv(B)$ if and only if $G$ is bipartite, where the ``if''-part of this equivalence follows from
Theorem 3.10 in~\cite{Luedtke.etal_2012_Someresultsstrength}. In contrast, without restricting the
signs of the coefficients bipartiteness does not help. The probabilistic argument in the proof of
Theorem~\ref{thm:asymptotics} also works for the complete bipartite graph with equal parts and
yields that in this setting almost all functions $b$ with coefficients in $\{1,-1\}$ have
$\mcgap[b](\vect x)\geqslant\tfrac{\sqrt n}{8}\chgap[b](\vect x)$ for $\vect
x=\left(1/2,\ldots,1/2\right)$.

\section*{Acknowledgments}
This research was supported by the ARC Linkage Grant no. LP110200524, Hunter Valley Coal Chain Coordinator (\href{http://www.hvccc.com.au}{hvccc.com.au}) and Triple Point Technology (\href{http://www.tpt.com}{tpt.com}).

We thank Jeff Linderoth and James Luedtke for fruitful discussions of the topics presented in this
paper, both during a visit of Jeff \annotateText{Linderoth} to Newcastle, Australia, and at the 22\textsuperscript{nd} ISMP
in Pittsburgh. We also thank Aleksandar Nikolov for \annotateText{pointing us to the ``old arguments by Spencer and Erd\H{o}s'' used in the proof of Theorem~\ref{thm:large_cuts} (see~\cite{__Combinatorialdiscrepancysystem}).}

 
\end{document}